\documentclass[12pt,reqno,a4paper]{amsart} 
\usepackage{graphicx}
\usepackage{amssymb,amsmath}
\usepackage{amsthm}
\usepackage{color,graphicx}
\usepackage{hyperref}
\usepackage{color}
\usepackage{mathabx}
\usepackage{subfigure}
\usepackage[pagewise]{lineno}

\usepackage{subfigure} 

\usepackage{verbatim} 

\usepackage{relsize} 

\newcommand{\be}{\begin{equation}}

\newcommand{\ee}{\end{equation}}

\newcommand{\R}{{\mathbb R}}

\newcommand{\Ldois}{{\mathcal{L}_2}}

\newcommand{\Ker}{{\rm \,Ker}}

\newcommand{\bigslant}[2]{{\raisebox{.1em}{$#1$}\left/\raisebox{-.1em}{$#2$}\right.}}

\numberwithin{equation}{section}
\numberwithin{figure}{section}

\newtheorem{theorem}{Theorem}[section]
\newtheorem{proposition}[theorem]{Proposition}
\newtheorem{remark}[theorem]{Remark}
\newtheorem{lemma}[theorem]{Lemma}

\newtheorem{definition}[theorem]{Definition}

\begin{document}
\vglue-1cm \hskip1cm
\title[existence and spectral stability for the 2D-NLS]{Existence and spectral stability of small-amplitude periodic waves for the 2D nonlinear focusing Schr\"odinger equation}

\begin{center}

\subjclass[2000]{35B10, 35B32, 35B35  }

\keywords{Existence of small-amplitude periodic waves, existence of periodic minimizers, Schr\"odinger equation, spectral stability.}

\maketitle

{\bf F\'abio Natali}

{Departamento de Matem\'atica - Universidade Estadual de Maring\'a\\
Avenida Colombo, 5790, CEP 87020-900, Maring\'a, PR, Brasil.}\\
{ fmanatali@uem.br}




\vspace{3mm}

\end{center}

\begin{abstract} The purpose of this paper is to establish the existence and spectral stability, with respect to perturbations of the same period, of double-periodic standing waves for the nonlinear focusing Schrödinger equation posed on the bi-dimensional torus. We first show that such double-periodic solutions can be constructed via local and global bifurcation theory, under the assumption that the kernel of the linearized operator around the equilibrium solution is one-dimensional. In addition, we prove that these local and global solutions minimize an appropriate variational problem, which enables us to derive spectral properties of the linearized operator about the periodic wave. Finally, we establish the spectral stability of small-amplitude periodic waves by applying the techniques developed in \cite{KapitulaKevrekidisSandstedeII} and \cite{KapitulaKevrekidisSandstedeI}.
\end{abstract}

\section{Introduction} 
\indent Consider the general bi-dimensional nonlinear focusing Schr\"odinger equation 
\begin{equation}\label{nls}
	iu_t+\Delta u +|u|^pu=0,
	\end{equation}
	where $p\in\mathbb{N}$, $t\in\mathbb{R}$, and $u=u(x,y,t)\in\mathbb{C}$ is a double $2\pi-$periodic function in the first two variables. This means that the pair $(x,y)\in\mathbb{R}\times\mathbb{R}$ can be considered in the bi-torus $\mathbb{T}\times\mathbb{T}$, where the functions are $2\pi-$periodic in both variables $x$ and $y$.\\ 
	\indent Regarding equation $(\ref{nls})$ defined on a periodic domain, a pioneering contribution was made by Bourgain in \cite{bourgain1}, where the author established both local and global results for the associated Cauchy problem in the more general periodic domain $\mathbb{T}^n = \mathbb{T} \times \cdots \times \mathbb{T}$, with $n \in \mathbb{N}$. The cornerstone of this seminal work is the identification of the sharp threshold for the index $\alpha \in \mathbb{R}$ such that, for initial data $u_0 \in H_{\rm per}^{\alpha}(\mathbb{T}^n)$, the solution $u(t)$ remains in $H_{\rm per}^{\alpha}(\mathbb{T}^n)$ for all $t$ in a suitable time interval. Additional contributions to this topic can be found in \cite{Cazenave}, \cite{erdogan}, and \cite{ginibre}.\\
	\indent  The important qualitative aspects concerning equation $(\ref{nls})$, which address spectral stability, modulational stability, and orbital stability, have significant contributions in the case $n = 1$ with the power nonlinearity $p = 2$. In fact, we have the work of Angulo \cite{Angulo}, where the author determined results on the orbital stability of positive and periodic standing waves. To achieve this, the author combined classical Floquet theory for the associated Hill operators around the periodic wave with the stability approaches in \cite{GrillakisShatahStraussI} and \cite{WeinsteinNLS}. Deconinck and Upsal \cite{DeconinckUpsal} used the integrability of equation $(\ref{nls})$, still in the case $p = 2$, to establish orbital stability results for the same waves with respect to subharmonic perturbations in the space of continuous bounded functions.\\ 
	\indent In Gustafson \textit{et al.} \cite{GustafsonLecozTsai}, the authors obtained periodic wave solutions that change sign for the case $p = 2$ using a variational method, and they proved spectral stability results with respect to perturbations of the same period $L > 0$ and orbital stability results in the space of anti-periodic functions with period $L/2$. The orbital stability of the same periodic waves was also determined by Natali \textit{et al.} in \cite{NMLP2021}. However, the authors have restricted the analysis over the Sobolev space $H_{\rm per}^1$ constituted by zero mean periodic functions.\\
	\indent Concerning the stability of small-amplitude periodic waves for $n = 1$, Gallay and H\u ar\u agu\c s  \cite{gallay1} considered the equation.
	\begin{equation}\label{focdef}
		iu_t + u_{xx} + \gamma |u|^2 u = 0,
	\end{equation}
	where $\gamma = 1$ and $\gamma = -1$ correspond to the focusing and defocusing cases, respectively. In both cases, the authors show that periodic standing waves are orbitally stable within the class of solutions having the same period. Spectral stability with respect to bounded or localized perturbations was also reported. In fact, they show that small-amplitude waves are stable in the defocusing case, but unstable in the focusing case.\\
	\indent For the case $p = 4$, Angulo and Natali in \cite{AN2} showed the existence of a unique critical frequency $c^{*} > 0$ such that all positive and periodic waves are orbitally stable for all $c \in \left(\frac{\pi^2}{L^2}, c^{*}\right)$ and orbitally unstable for all $c \in (c^{*}, +\infty)$. The approach used by the authors was based on \cite{GrillakisShatahStraussI}. Regarding periodic waves that change their sign, Moraes and de Loreno in \cite{gabriel-guilherme} established orbital instability results using the abstract approach in \cite{ShatahStraussbook}, which improves the instability results found in \cite{grillakis2}.
	 \\
\indent Next, we describe the main topics of our paper. Standing wave solutions of the form $u(x,y,t)=e^{ict}\varphi(x,y)$ can be determined by substituting this special form in $(\ref{nls})$ to obtain the elliptic equation
	\begin{equation}\label{ode}
		-\Delta \varphi+c\varphi-\varphi^{p+1}=0,\end{equation}
		where $c\in\mathbb{R}$ is called frequency of the wave $\varphi$. In addition, the bi-dimensional NLS equation  admits formally the following conserved quantities, which are useful in our spectral stability analysis:
		\begin{equation}\label{Eu}
			E(u) = \frac{1}{2} \int_{\mathbb{T}\times \mathbb{T}} |\nabla u|^2dxdy - \frac{1}{p+2} \int_{\mathbb{T}\times \mathbb{T}} |u|^{p+2} dxdy,
		\end{equation}
		\begin{equation}\label{Fu}
			F(u)=\frac{1}{2}\int_{\mathbb{T}\times \mathbb{T}}|u|^2dxdy.
		\end{equation}
		\begin{remark}
			There are other two important conserved quantities associated with the equation $(\ref{nls})$ (see \cite[Section 2]{sulem}) posed in $\mathbb{R}^n$. The first one is the linear momentum given by
			$$\overrightarrow{P}(u)={\rm Im}\int_{\mathbb{R}^n}\bar{u}\nabla u\ d\overrightarrow{x}.$$
			This conserved quantity is useful for studying the orbital stability of traveling waves of equation $(\ref{nls})$, especially when the equation is considered in $\mathbb{R}^n$. In the periodic case, the linear momentum can still be considered in the study of traveling wave stability, but certain restrictions must be imposed (see \cite[page 3]{NP1}). Another important conserved quantity is the angular momentum, expressed by
			$$\overrightarrow{A}(u)={\rm Im}\int_{\mathbb{R}^n}\bar{u}(\overrightarrow{x}\wedge\nabla u)\ d\overrightarrow{x},$$
		where $\overrightarrow{x}=(x_1,x_2,\cdots,x_n)$ and $\overrightarrow{x}\wedge\nabla u$ denotes the vector cross product in $\mathbb{R}^n$. In our context, the conservation quantity $\overrightarrow{A}$ does not play an effective role in the stability analysis of periodic standing waves. In addition, the vector cross product $\overrightarrow{x}\wedge\nabla u$ is not well defined in the periodic context due the lack of spatial decay to zero at infinite.
		\end{remark}
		\indent A suitable and new way to construct periodic real valued solutions for the elliptic equation \eqref{ode} can be determined by using the local and global bifurcation theory in \cite{buffoni-toland} (see also \cite{BD}). First, we prove the existence of small-amplitude periodic solutions for $c>\frac{2}{p}$ and close to the bifurcation point $\frac{2}{p}$. After that, we give sufficient conditions to extend parameter $c$ to the whole interval $\left(\frac{2}{p},+\infty\right)$ by constructing a symmetric periodic continuous function $c \in \left( \frac{2}{p}, +\infty \right) \mapsto \varphi_{c} \in H_{\rm per,s}^{2}$,
		where for $r\geq0$, we define $H_{\rm per,s}^{r}=\{f\in H_{\rm per}^2;\ f(-x,y)=f(x,y)\ \mbox{and}\ f(x,-y)=f(x,y)\ \mbox{a.e. in}\ \mathbb{T}\times\mathbb{T}\}$. Important to mention that small-amplitude periodic waves have been used to show the existence and spectral/orbital stability of periodic solutions associated with several evolutionary equations. As examples, we can cite \cite{natali1}, \cite{natali2}, \cite{natali3}, and \cite{natali4}.

Another way to determine the existence of periodic waves $\varphi$ that solves $(\ref{ode})$ is to find a minimizer of the following complex problem,
	\begin{equation}
		\label{infB2}
		q = \inf_{u\in Y_\lambda} B_c(u), \quad
		B_c(u) = \frac{1}{2}\int_{\mathbb{T}\times\mathbb{T}}|\nabla u|^2+c|u|^2dxdy,
	\end{equation}
	in the constrained set
	\begin{equation}
		\label{Y-constraint2}
		Y_{\lambda} = \left\{u\in H_{\rm per,s}^1;\ \int_{\mathbb{T}\times\mathbb{T}} |u|^{p+2} dxdy = \lambda \right\}.
	\end{equation}
	\indent Once we determine a complex function $\Psi$ that solves the complex periodic boundary problem 
	
	\begin{equation}\label{complexode}
		-\Delta\psi+c\Psi-|\Psi|^p\Psi=0,
		\end{equation}
	we need to establish the existence of $\theta\in\mathbb{R}$ and a smooth positive and periodic real valued function $\varphi$ such that $\Psi(x,y)=e^{i\theta}\varphi(x,y)$ for all $(x,y)\in\mathbb{T}\times \mathbb{T}$, where $\varphi$ solves equation $(\ref{ode})$. The arguments to do so are an adaptation for the periodic case of the approach in \cite[Section 8]{Cazenave}. In addition, we also prove that $c > \frac{2}{p}$, and since $c \rightarrow \frac{2}{p}^{+}$, we find that the small-amplitude periodic solutions can be considered minimizers of the problem $(\ref{infB2})$.\\
		\indent Now, we present how to obtain the spectral stability of periodic waves with respect to perturbations of the same period. Indeed, in order to improve the readers' understanding, we split equation $(\ref{nls})$ into its real and imaginary parts. This means that we can consider $U = (P,Q)$ and the system $(\ref{nls})$ becomes,
		
			\begin{equation}\label{equisys}
			\left\{\begin{array}{cccc}
				P_t+\Delta Q+Q(P^2+Q^2)^{\frac{p}{2}}=0,\\
				-Q_t+\Delta P+P(P^2+Q^2)^{\frac{p}{2}}=0.
			\end{array}\right.
		\end{equation}
		Quantities \eqref{Eu} and \eqref{Fu}, in this new scenario, are given by
		\begin{equation}\label{hamiltocons1}
			E(U)=E(P,Q)=\frac{1}{2}\int_{\mathbb{T}\times\mathbb{T}} \bigg(P_x^2+P_y^2+Q_x^2+Q_y^2-\frac{2}{p+2}(P^2+Q^2)^{\frac{p+2}{2}}\bigg)\,dxdy,
		\end{equation}
		and
		\begin{equation}\label{mass1}
			F(U)=F(P,Q)=\frac{1}{2}\int_{\mathbb{T}\times\mathbb{T}} (P^2+Q^2)\,dxdy.
		\end{equation}
		As a consequence, \eqref{equisys} or, equivalently, \eqref{nls} can be written as
		\begin{equation}\label{hamiltonian}
			\frac{d}{dt}U(t)=J E'(U(t)),
		\end{equation}
		where $E'$ represents the Fr\'echet derivative of $E$ with respect to $U$, and
		\begin{equation}\label{J}
			J=\left(\begin{array}{cccc}
				0 & -1\\
				1 & 0
			\end{array}\right).
		\end{equation}

\indent To set our problem, let us consider the stationary solution
$\Psi = (\varphi,0)$ and the general perturbation
\begin{equation}\label{U-1}
	U(x,t) = e^{i c t} (\Psi(x,y) + W(x,y,t))
\end{equation} where $W =(W_1,W_2)\in \mathbb{R}^2$ and $\beta\in\mathbb{R}$. Substituting \eqref{U-1} into \eqref{hamiltonian} and neglecting all the nonlinear terms, we get the following linearized equation:
\begin{equation}\label{spectral}
	\frac{d}{dt} W(x,y,t) = J \mathcal{L} W(x,y,t),
\end{equation}
where $J$ is given by $(\ref{J})$. The linear operator $\mathcal{L}$ in $(\ref{spectral})$ is given by
\begin{equation}\label{matrixop}
 \mathcal{L}= \begin{pmatrix}
\mathcal{L}_1 & 0 \\
0 & \mathcal{L}_2
\end{pmatrix},
\end{equation}
where
\begin{equation}\label{L1L2}
\mathcal{L}_1=-\Delta+c-(p+1)\varphi^p
\qquad \text{and} \qquad
\mathcal{L}_2=-\Delta+c-\varphi^p.
\end{equation}
Both operators $\mathcal{L}_1$ and $\mathcal{L}_2$ are self-adjoint  and they are deﬁned in $L^2_{\rm per,s}$ with dense domain $H^{2}_{\rm per,s}$. \\
\indent To define the concept of spectral stability within our context, we need to substitute the growing mode solution of the form $W(x,y,t)=e^{\lambda t}w(x,y)$ into the linear equation \eqref{spectral} to obtain the following spectral problem
\begin{equation}\label{JL}
	J \mathcal{L} w = \lambda w.
\end{equation} 
\indent The definition of spectral stability in our context reads as follows.
\begin{definition}\label{def-spectralstability}
	The stationary wave $\Psi$ is said to be spectrally stable by periodic perturbations that have the same period as the standing wave solution if $\sigma(J \mathcal{L}) \subset i \mathbb{R}$. Otherwise, if there exists at least one eigenvalue $\lambda$ associated with the operator $J \mathcal{L}$ that has a positive real part, $\Psi$ is said to be spectrally unstable.
\end{definition}

\indent Let ${\rm n}(\mathcal{A})$ and ${\rm z}(\mathcal{A})$  be the number of negative eigenvalues and the dimension of the kernel of a certain linear operator $\mathcal{A}$. In our paper, a prior understanding of these non-negative numbers is essential for obtaining the spectral stability result. To do so, we need to use the methods developed by Kapitula, Kevrekidis and Sandstede in \cite{KapitulaKevrekidisSandstedeII} and \cite{KapitulaKevrekidisSandstedeI}.\\
\indent Indeed, if ${\rm z}(\mathcal{L})=n$, consider $\{\Theta_l\}_{l=1}^n \subset {\rm Ker}(\mathcal{L})$ a linearly independent set and let $V$ be the $n\times n$ matrix whose entries are given by 
\begin{equation}\label{V}
	V_{jl} = ( \mathcal{L}^{-1} J \Theta_j, J \Theta_l )_{{L}^2_{\rm per}},
\end{equation}
where $1\leq j,l\leq n$. If $V$ is invertible, the formula
\begin{equation}\label{krein}
	k_r + k_c + k_{-} = {\rm n}(\mathcal{L}) - {\rm n}(V),
\end{equation}
is given in \cite{KapitulaKevrekidisSandstedeII} and the left-hand side of $(\ref{krein})$ is called the hamiltonian Krein index\footnote{If $V$ is not invertible, the formula in $(\ref{krein})$ reduces to $k_r + k_c + k_{-} = {\rm n}(\mathcal{L}) - {\rm n}(V)- {\rm z}(V)$ and the spectral stability can be determined by considering the dimension of the kernel of $V$}. Regarding operator $J\mathcal{L}$ in $(\ref{JL})$, let $k_r$ be the number of real-valued and positive eigenvalues (counting multiplicities). The number $k_c$ denotes the number of complex-valued eigenvalues with a positive real part and $k_-$ is the number of pairs of purely imaginary eigenvalues with negative Krein signature of $J\mathcal{L}$. Since $k_c$ and $k_-$ are always even numbers, we obtain that if the right-hand side in \eqref{krein} is an odd number, then $k_r \geq 1$ and we have automatically the spectral instability. Moreover, if the difference ${\rm n}(\mathcal{L}) - {\rm n}(V)$ is zero, then $k_c = k_- = k_r = 0$ which implies the spectral stability.\\
\indent In our case, since the small-amplitude periodic waves solve the minimization problem $(\ref{infB2})$, we deduce that ${\rm n}(\mathcal{L}_1)=1$, ${\rm n}(\mathcal{L}_2)=0$, and ${\rm z}(\mathcal{L}_2)=1$. In addition, since $\mathcal{L}$ is given by $(\ref{matrixop})$ and $J$ is given by $(\ref{J})$, we obtain, in the case where $\Ker(\mathcal{L}_1)=[v_1]$ that $V_{12}=V_{21}=0$, $V_{11}= ( \mathcal{L}_2^{-1}v_1, v_1)_{{L}^2_{\rm per}}$, and $V_{22}=( \mathcal{L}_1^{-1}\varphi, \varphi)_{{L}^2_{\rm per}}$. The fact small-amplitude periodic solution $\varphi$ depends smoothly on the parameter $c$ enables to conclude that $\varphi\in {\rm Range}(\mathcal{L}_1)$. This means that there exists $h=-\frac{d\varphi}{dc}\in H_{\rm per,s}^2$, such that $\mathcal{L}_1h=\varphi$. Consequently, we have $(v_1,\varphi)_{L_{\rm per}^2}=(v_1,\mathcal{L}_1\chi)_{L_{\rm per}^2}=(\mathcal{L}_1v_1,\chi)_{L_{\rm per}^2}=0$. This last fact implies $V_{11}>0$, since ${\rm z}(\mathcal{L}_2)=1$ and ${\rm n}(\mathcal{L}_2)=0$. Therefore, since for small-amplitude periodic waves, we can prove the inequality $V_{22} < 0$, we obtain that $\det(V) = V_{11} V_{22} < 0$. This, in turn, indicates that $V$ has exactly one negative eigenvalue, that is, $n(V) = 1$. Since $n(\mathcal{L}) = 1$, we deduce spectral stability. Now, if $\Ker(\mathcal{L}_1)$ is trivial, we obtain that $V=( \mathcal{L}_1^{-1}\varphi, \varphi)_{{L}^2_{\rm per}}$. Since in this case, $( \mathcal{L}_1^{-1}\varphi, \varphi)_{{L}^2_{\rm per}}<0$ for small-amplitude periodic waves, we also conclude the spectral stability.



 Just to fix the basic notation used throughout the paper. Given $s\in\R$, by $H_{\rm per}^s=H^s(\mathbb{T}\times\mathbb{T})$ we denote the usual Sobolev space of real or complex functions. In particular $H_{\rm per}^0\cong L_{\rm per}^2$. The scalar product in $H_{\rm per}^s$ will be denoted by $(\cdot,\cdot)_{H_{\rm per}^s}$. The notation $a \sim b$ means that $a$ is sufficiently close to $b$; that is, for all $\varepsilon > 0$, we have $|a - b| < \varepsilon$. Similarly, the notation $a \sim b^{+}$ means that $a$ is sufficiently close to $b$ from the right, which means that for all $\varepsilon > 0$, we have $a - b < \varepsilon$ and $a > b$. \\
 \indent Next, we present our main result, which summarizes the discussion contained in the second paragraph of this page.

\begin{theorem}\label{teo-1}
	 Let $p\geq1$ be a fixed integer. Consider the small-amplitude periodic wave solution $\varphi\in H^1_{\rm per,s}$ of the equation \eqref{ode} which solves the minimization problem in $(\ref{infB2})$. The wave $\Psi = (\varphi, 0)$ is spectrally stable for all $c\sim\frac{2}{p}^{+}$ in the sense of Definition $\ref{def-spectralstability}$, in the space $L_{\rm per,s}^2$.\end{theorem}
\begin{remark}\label{remexpl} For the reader’s better understanding, we provide all the necessary details to establish the existence of small-amplitude periodic solutions via bifurcation theory, along with their global continuation for all values of $c > \frac{2}{p}$, specifically in the case $p = 1$. For any integer $p\geq 1$, the arguments are similar (see Remark $\ref{expsol3}$). We believe that establishing the results in a particular case before extending them to the general setting is a useful approach. The proof of Theorem $\ref{teo-1}$ in Section 3 is presented for any integer $p \geq 1$.
\end{remark}
\indent Our paper is organized as follows: Section 2 is devoted to present the existence of periodic waves using local/global bifurcation and using variational techniques that can be connected with both local and global solutions. In Section 3, we present the spectral stability of small-amplitude periodic waves constructed in Section 2. Finally, in Section 4, we establish some important remarks concerning the construction of small-amplitude periodic waves in suitable spaces.\\

			\section{Existence of periodic waves and spectral analysis}
			 In order to obtain the existence of small-amplitude periodic waves for the equation $(\ref{ode})$, we need to define some facts regarding Fredholm operators:
			
			\begin{definition}\label{def1} Let $X$ be a Banach space. An unbounded operator $L:D(L)\subset X\rightarrow Y$ is a Fredholm operator if ${\rm{Range}}(L)$ is closed and ${\rm{z}}(L)$ and ${\rm{c}}(L)$ are both finite. Here, ${\rm c}(L)$ indicates the dimension of ${\rm Coker}(L)$. 
			\end{definition}
			
			\begin{remark}
				Just to make clear for the readers: ${\rm Coker}(L)$ denotes the quotient space given by ${\rm Coker}(L)=\bigslant{Y}{{\rm Range}(L)}$.		
			\end{remark}
			
			\begin{definition}\label{def2} The index of an unbounded Fredholm operator $L:D(L)\subset X\rightarrow Y$ is given by ${\rm ind}(L)={\rm{z}}(L)-{\rm{c}}(L)\in \mathbb{Z}$. A Fredholm operator is of index zero if ${\rm ind}(L)=0$.		
				\end{definition}
			\indent Next result is very useful for our purposes in order to obtain a relation between ${\rm{c}}(L)$ and ${\rm{z}}(L)$:

			\begin{lemma}\label{lema3} Let $H$ be a real Hilbert space and $K\subset H$ a closed subspace. It follows that, $$\bigslant{H}{K}\cong K^{\bot},$$
			where the notation $A \cong B$ indicates that $A$ and $B$ are isomorphic. Therefore, if both $A$ and $B$ are finite dimensional, they have the same dimension.

			\end{lemma}				
			\begin{proof} Let us define $T:\bigslant{H}{K}\rightarrow K^{\bot}$ given by $T(u+K)=u-P_Ku$, where $P_K$ is the orthogonal projection from $H$ onto the closed subspace $K$. It is well known that for any $u\in H$, we obtain $P_Ku\in K$ and $u-P_Ku\in K^{\bot}$, that is, $T$ is well-defined. In addition, since $||T(u+K)||_H=||u-P_Ku||$, we obtain by Pythagorean theorem  $||u||_{H}^2=||P_Ku||_H^2+||u-P_Ku||_H^2=||P_Ku||_H^2+||T(u+K)||_H^2.$
				The equality implies $||T(u+K)||_H^2=||u||_{H}^2-||P_Ku||_H^2\leq ||u||_{H}^2$, and thus, $T$ is a bounded operator. $T$ is an one-to-one operator since for $T(u+K)=0$, we have $u=P_Ku$ and this fact automatically implies $u\in K$, that is, $u+K=0$. To see that $T$ is onto, we consider $v\in K^{\bot}$. By the definition of orthogonal projection from $H$ onto the closed subspace $K$, there exists $u\in H$ such that $v=u-P_Ku$, and $T$ is onto as desired.
				
			\end{proof}
			
			\begin{remark}\label{obs2} We can provide a new way to view Definitions $\ref{def1}$ and $\ref{def2}$ for a Hilbert space $H$ and an unbounded self-adjoint linear operator $L:D(L)\subset H\rightarrow H$ with closed range. In fact, since $L$ is self-adjoint with closed range, we have by Lemma $\ref{lema3}$ that $\bigslant{H}{{\rm Range}(L)}= \bigslant{H}{\Ker(L)^{\bot}}\cong\Ker(L)^{{\bot}{\bot}}=\Ker(L)$. Therefore, we can conclude in this specific case, if ${\rm z}(L)$ is finite, that $L$ is always a Fredholm operator of index zero.	
				
			\end{remark}
			\indent Now, we have the first local bifurcation theorem to obtain the existence of small-amplitude periodic waves.
			
			\begin{proposition}\label{TCR} Suppose that $X$ and $Y$ are Banach spaces, that $F: X\times \mathbb{R}\rightarrow Y$ is of class $C^k$, $k\geq2$, and that $F(x_0,\lambda_0)=0\in Y$ for some $(x_0,\lambda_0) \in X\times \mathbb{R}$. Suppose also that
				\begin{enumerate}
					\item $\partial_gF(g,\lambda)$ is a Fredholm operator of index zero when $F(g,\lambda)=0$ for all $(g,\lambda)\in U$. Here $U\subset  X\times \mathbb{R}$ denotes a subset.
					\item For some $(x_0,\lambda_0)\in X\times \mathbb{R}$, $\ker(L_{x_0,\lambda_0})$ is one dimensional, where $L_{x_0,\lambda_0}=\partial_gF(x_0,\lambda_0)$. This means that
					$\Ker(L_{x_0,\lambda_0})=\{h\in X;\ h=\alpha h_0\ \mbox{for some}\ \alpha\in\mathbb{R}\},$ $h_0\in X\backslash\{0\}$.
					\item The transversality condition holds: $\partial_{g,\lambda}^2F[(x_0,\lambda_0)](h_0,1)\notin {\rm{Range}}(L_{x_0,\lambda_0})$.
				\end{enumerate} 
				Then, there exists $\epsilon>0$ and a branch of solutions 
				$\{(\chi(s),\Lambda(s));\ s\in\mathbb{R},\ |s|<\epsilon\}\subset  X\times \mathbb{R},$
				such that $\Lambda(0)=\lambda_0$ and $\chi(0)=x_0$. In addition, we have
				\begin{itemize}
					\item $F(\chi(s),\Lambda(s))=0$ for all $s$ with $|s|<\epsilon$.
					\item functions $s\mapsto\Lambda(s)$ and $s\mapsto \chi(s)$ are of class $C^{k-1}$, and $\chi$ is of class $C^{k-2}$, on $(-\epsilon,\epsilon)$.
					\item there exists an open set $U_0\subset X\times\mathbb{R}$ such that $(x_0,\lambda_0)\in U_0$ and 
					$\{(g,\lambda)\in U_0;\ F(g,\lambda)=0,\ g\neq0\}=\{(\chi(s),\Lambda(s));\ 0<|s|<\epsilon\}.$
					\item If $F$ is analytic, $\chi$ and $\Lambda$ are analytic functions on $(-\epsilon,\epsilon)$.					
					\end{itemize}
			\end{proposition}
			\begin{proof} See \cite[Theorem 8.3.1]{buffoni-toland} and also \cite[Page 114]{buffoni-toland}.
				
			\end{proof}

			To prove the existence of small-amplitude periodic waves, we need to establish some basic facts. First, we consider the symmetric space:			
			$$L_{\rm per,s}^2=\{f\in L_{\rm per}^2;\ f(-x,y)=f(x,y)\ \mbox{and}\ f(x,-y)=f(x,y)\ \mbox{a.e. in}\ \mathbb{T}\times\mathbb{T}\}.$$
			
			The subspace $L_{\rm per,s}^2\subset L_{\rm per}^2$ is endowed with the inner product and norm of $L_{\rm per}^2$ and it is closed subspace contained in $L_{\rm per}^2$. In addition, if 			
			$$L_{\rm per,e}^2=\{f\in L_{\rm per}^2;\ f(-x,-y)=f(x,y)\ \mbox{a.e. in}\ \mathbb{T}\times\mathbb{T}\},$$
			indicates the subspace of $L_{\rm per}^2$ constituted by even periodic functions, we obtain that $L_{\rm per,s}^2\subset L_{\rm per,e}^2$. Just to make clear that the last inclusion  is strict, we consider $f(x,y)=\sin(x)\sin(y)$. We see that $f\in L_{\rm per,e}^2$, but $f$ is not an element of $L_{\rm per,s}^2$.
			
			\begin{lemma}\label{lema2} For $r\in\mathbb{R}$, consider $L_{r}=-\Delta +r$ as the linear operator defined in $L_{\rm per, e}^2$ with domain $H_{\rm per, e}^2$. We have that $\Ker(L_r)={\rm span}\{\cos(kx)\cos(jy),\sin(kx)\sin(jy)\}$ for all $k,j\in \mathbb{Z}$ satisfying $r^2=-k^2-j^2$.				
				\end{lemma}
				
			\begin{proof}
				Consider $u=u(x,y)$, a smooth periodic function defined in $\mathbb{T}\times\mathbb{T}$. Applying the Fourier transform to the equation $L_ru=0$, we obtain 
				\begin{equation}\label{fou1}
					((k^2+j^2)+r)\widehat{u}(k,j)=0.					
					\end{equation}
			For a fixed value of $r\in\mathbb{R}$, the existence of non-zero solutions is guaranteed by $(\ref{fou1})$ if $k,j\in\mathbb{Z}$ satisfy $r=-k^2-j^2$. This implies that $u_1(x,y)=\cos(kx+jy)$ and $u_2(x,y)=\cos(kx-jy)$ are the only even periodic functions defined over $\mathbb{T}\times\mathbb{T}$ which belong to the kernel of $L_r$. Using that $\cos(kx\pm jy)=\cos(kx)\cos(jy)\mp\sin(kx)\sin(jy)$ to simplify the notation, we obtain  the desired result.
				
			\end{proof}
			
			\subsection{Existence of periodic solutions via local and global bifurcation for the case $p=1$}
			
			\begin{proposition}\label{propstokes1}
				There exists $a_0 > 0$ such that for all $a \in (0,a_0)$ there is an even local periodic solution $\varphi$ for the problem \eqref{ode}. Small-amplitude periodic waves can be expressed by the following expansion
				\begin{equation}\label{varphi-stokes31}\begin{array}{lllll}
					\varphi(x,y) &=&\displaystyle 2+a\cos(x)\cos(y)\\\\
					&+&\displaystyle\frac{1}{8}a^2\left(-\frac{7}{6}+\left(\cos(2x)+\cos(2y)\right)+\frac{1}{3}\cos(2x)\cos(2y)\right)\\\\
					&+&\displaystyle \frac{1}{384}a^3\left(7\left(\cos(3x)\cos(y)+\cos(3y)\cos(x)\right)+\frac{1}{2}\cos(3x)\cos(3y)\right)\\\\
					&+&\mathcal{O}(a^4).
				\end{array}\end{equation}
				The frequency $c$ is expressed as	
				
				\begin{equation}\label{varphi-stokes3}
					c = 2-\frac{1}{48}a^2+\mathcal{O}(a^4).
				\end{equation}
			\end{proposition}
			
			\begin{proof} We shall give all the steps how to prove the existence of small-amplitude periodic waves using Proposition $\ref{TCR}$. In fact, let $\mathsf{F}: H_{\rm per,s}^{2} \times \mathbb{R} \rightarrow L_{\rm per,s}^2$ be the smooth map defined by
				\begin{equation}\label{F-lyapunov}
					\mathsf{F}(g,c) = -\Delta g+cg-g^2.
				\end{equation}
				 We see that $\mathsf{F}(g,c) = 0$ if, and only if, $g \in H_{\rm per,s}^{2}$ satisfies \eqref{ode} with corresponding frequency $c \in \mathbb{R}$. The Fr\'echet derivative of the function $\mathsf{F}$ with respect to the first variable at the equilibrium point $(c_0,c_0)$ is then given by
				\begin{equation}\label{Dg}
					D_g \mathsf{F}(c_0, c_0) f = (-\Delta - c_0) f.
				\end{equation}
				
				The nontrivial kernel of $D_g\mathsf{F}(c_0, c_0)$ is determined by functions $h \in H_{\rm per,s}^{2}$ such that
				\begin{equation}
					\widehat{h}(k,j) (k^2+j^2-c_0) = 0.
				\end{equation}
				We see that $D_g \mathsf{F}(c_0, c_0)$ has one-dimensional kernel if, and only if, $c_0 = k^2+j^2$. In this case, we have
				\begin{equation}
					{\rm Ker}D_g \mathsf{F}(c_0, c_0) = {\rm span}\{\tilde{\varphi}_{k,j}\}.
				\end{equation}
				By Lemma $\ref{lema2}$, it follows that $\tilde{\varphi}_{k,j}(x,y) = \cos(kx)\cos(jy)$. In addition, since $	D_g \mathsf{F}(c_0, c_0)$ is a self-adjoint operator defined in $L_{\rm per, s}^2$ with domain in $H_{\rm per, s}^2$, we also have that the transversality condition $\cos(kx)\cos(jy)\notin  {\rm Ker}D_g \mathsf{F}(c_0, c_0)^{\bot}={\rm Range}D_g \mathsf{F}(c_0, c_0)$ is satisfied. 
				
				In the previous analysis, let us consider $k=j=1$. We obtain that $c_0=2$ and this fact enables us to define the set $$\mathcal{S} = \{(g, c) \in U=H_{\rm per,s}^2\times \left(2,+\infty\right);\ \mathsf{F}(g, c) = 0\}.$$ 
			 Let $(g, c) \in \mathcal{S}$ be a real solution of $\mathsf{F}(g, c) = 0$. We see by Remark $\ref{obs2}$ that
			 $D_gF(g,c)=-\Delta+c-2g$ is a Fredholm operator of index zero. In fact, in order to simplify the notation, let us consider $\mathcal{P}=D_gF(g,c)$. First of all, we see that $\mathcal{P}$ is clearly a self-adjoint operator. Thus,
			 $\sigma(\mathcal{P})=\sigma_{\rm disc}(\mathcal{P})\cup \sigma_{\rm ess}(\mathcal{P})$,  where $\sigma(\mathcal{P})$
			 denotes the spectrum of  $\mathcal{P}$, and  $\sigma_{\rm disc}(\mathcal{P})$ and $\sigma_{\rm ess}(\mathcal{P})$
			 denote, respectively, the discrete and essential spectra of  $\mathcal{P}$. Since $H^2_{\rm per,s}$
			 is compactly embedded in $L^2_{\rm per,s}$, the operator $\mathcal{P}$ has compact resolvent.
			 Consequently, $\sigma_{\rm ess}(\mathcal{P})=\emptyset$ and $\sigma(\mathcal{P})=\sigma_{\rm disc}(\mathcal{P})$ consists
			 of isolated eigenvalues with finite algebraic multiplicities (see e.g., \cite[Section III.6]{kato}). Thus, if $0$ is an eigenvalue for $\mathcal{P}$, we obtain that the dimension of 
			 $\Ker\mathcal{P}$ is finite. Therefore, $\mathcal{P}$ is a Fredholm operator of index zero. Now, if $0$ is not an eigenvalue, thus ${\rm z}(\mathcal{P})=0$ and we also obtain that $\mathcal{P}$ is a Fredholm operator of index zero as desired. \\
			 \indent The local bifurcation theory established in Proposition $\ref{TCR}$ guarantees the existence of an open interval $I$ to the right of and sufficiently close to $2$, an open ball $B(0,r) \subset H_{\rm per,s}^{2}$ for some $r>0$ and a  smooth mapping
			 \begin{equation}\label{localcurve}c \in I \mapsto \varphi= \varphi_c \in B(0,r) \subset H_{\rm per,s}^{2},\end{equation}
			 where $\mathsf{F}(\varphi,c) = 0$ for all $c \in I$ and $\varphi\in B(0,r)$.
			 
			 Next, we determine the explicit formulas $(\ref{varphi-stokes31})$ and $(\ref{varphi-stokes3})$. Indeed, let us consider the classical Stokes expansions given by
			 \begin{equation}\label{stokes1}\varphi(x,y) =2+\sum_{n=1}^{+\infty}a^n\varphi_n(x,y)\ \ \mbox{and}\ \ c = 2+ \sum_{n=1}^{+\infty}c_{2n}a^{2n}.\end{equation}
			 From $(\ref{ode})$ and $(\ref{stokes1})$, we obtain that $c_{2n}$ and $\varphi_n$ are uniquely determined since they satisfy the following recurrence relations
			 \begin{equation}\label{recurrence}\left\{\begin{array}{llllll}
			 	\mathcal{O}(a)\ :\ \ -\Delta\varphi_1-2\varphi_1=0,\\
			 	\mathcal{O}(a^2):\ \ - \Delta \varphi_2-2\varphi_2+2c_2-\varphi_1^2=0,\\
			 	\mathcal{O}(a^3):\ \ -\Delta\varphi_3-2\varphi_3+c_2\varphi_1-2\varphi_1\varphi_2=0.
			 	\end{array}\right.\end{equation}
			\indent As we have already determined above, $\varphi_1(x,y)=\cos(x)\cos(y)$ satisfies the equation containing the term $\mathcal{O}(a)$. Solving the inhomogeneous equation for $\mathcal{O}(a^2)$, we obtain
			\begin{equation}\label{stokes2}\varphi_2(x,y)=c_2-\frac{1}{8}+\frac{1}{8}\left(\cos(2x)+\cos(2y)\right)+\frac{1}{24}\cos(2x)\cos(2y),
			\end{equation} 
			where $c_2$ is a constant to be determined. The inhomogeneous equation at $\mathcal{O}(a^3)$ admits a solution $\varphi_3$ if, and only if, the right-hand side is orthogonal to $\varphi_1$, which selects uniquely the
			correction $c_2=-\frac{1}{48}$. In this case, $\varphi_3$ is given by
			 \begin{equation}\label{stokes3}\varphi_3(x,y)=\frac{7}{384}\left(\cos(3x)\cos(y)+\cos(3y)\cos(x)\right)+\frac{1}{768}\cos(3x)\cos(3y).
			 \end{equation} 
			 This finishes the proof of the proposition.	
				
			\end{proof}

\begin{remark}\label{remunique}
It is important to mention that the solution obtained in Proposition $\ref{propstokes1}$ is unique, up to the parametrization given by the bifurcation parameter. This uniqueness is guaranteed by the fact that the Lyapunov–Schmidt reduction relies on the application of the implicit function theorem in its proof (see the proof of Theorem 8.3.1 in \cite{buffoni-toland}). In the case of a one-dimensional kernel, bifurcation occurs along a single branch of solutions. The implicit function theorem then ensures the existence of a unique smooth curve of solutions bifurcating from the constant solution. This curve is parameterized by a small parameter, denoted by $a$ as in $(\ref{varphi-stokes31})$, which controls the amplitude of the bifurcating solution.

\end{remark}

\begin{proposition}\label{propstokes2}

	The local solution obtained in Proposition $\ref{propstokes1}$ is global, that is, $\varphi$ exists for all $c>2$. In addition, the pair $(\varphi, c) \in H_{\rm per,s}^2\times (2,+\infty) $ is continuous in terms of the parameter $c > 2$ and it satisfies \eqref{ode}.
\end{proposition}

\begin{proof}
	To obtain that the  local curve \eqref{localcurve} extends to a global one, we need to prove  that every bounded and closed subset $\mathcal{R}\subset\mathcal{S}$ is a compact set on $H_{\rm per,s}^{2}\times (2,+\infty)$. To this end, we want to prove that $\mathcal{R}$ is sequentially compact, that is, if $\{(g_n,c_n)\}_{n\in\mathbb{N}}$ is sequence in $\mathcal{R}$, then there exists a subsequence of $\{(g_n,c_n)\}_{n\in\mathbb{N}}$ that converges to a point in $\mathcal{R}$. Let $\{(g_n,c_n)\}_{n\in\mathbb{N}}$ be a sequence in $\mathcal{R}$. We obtain a subsequence with the same notation such that
	
	$$c_n\rightarrow c\ \ \ \mbox{in}\ \left(2,+\infty\right),$$
	and
	$$g_n\rightharpoonup g\ \ \ \mbox{in}\ H_{\rm per,s}^2,$$
	as $n\rightarrow +\infty$. If $c=2$, we obtain from the expression of $c$ given by $(\ref{varphi-stokes3})$, in a neighbourhood of $2$ to the right that, $a_n\rightarrow 0$ as $n\rightarrow +\infty$. Therefore, the solution $g_n$ has the form in $(\ref{varphi-stokes31})$ for each $n\in\mathbb{N}$, and it satisfies $g_n\rightarrow 2$ in $H_{\rm per,s}^2$. Hence, the result is proved, but the constant solution is not interesting for our purposes. Now, if $c>2$, we obtain that $(g_n,c_n)\in \mathcal{S}$ and the pair satisfies
	\begin{equation}\label{estI}
		\Delta g_n=c_ng_n-g_n^2.
	\end{equation}
	\indent Next, since $H_{\rm per,s}^2$ is a Banach algebra, we obtain that $g_n^2\in H_{\rm per,s}^2$ for all $n\in\mathbb{N}$. Next, we can express $\Delta g_n^2$ as
	\begin{equation}\label{estII}
	\Delta g_n^2=2\left[\left(\frac{\partial g_n}{\partial x}\right)^2+\left(\frac{\partial g_n}{\partial y}\right)^2\right]+2g_n\Delta g_n=2\left[\left(\frac{\partial g_n}{\partial x}\right)^2+\left(\frac{\partial g_n}{\partial y}\right)^2\right]+2c_ng_n^2-2g_n^3. 
	\end{equation}
	Using the Sobolev embedding $H_{\rm per,s}^1\hookrightarrow L_{\rm per,s}^4$ and the fact that $H_{\rm per,s}^2$ is a Banach algebra, we conclude that the right-hand side of $(\ref{estII})$ defines a bounded sequence in $L_{\rm per,s}^2$, that is, $\{\Delta g_n^2\}_{n\in\mathbb{N}}$ is bounded in $L_{\rm per,s}^2$. By $(\ref{estI})$ and $(\ref{estII})$, we then obtain that $\Delta^2 g_n$ exists and it defines a bounded sequence in $L_{\rm per,s}^2$. That is, $\{g_n\}_{n\in\mathbb{N}}$ is a bounded sequence in $H_{\rm per,s}^4$. By the compact embedding $H_{\rm per,s}^4\hookrightarrow H_{\rm per,s}^2$ we obtain, modulus a subsequence, that $$g_n\rightarrow g\ \ \ \mbox{in}\ H_{\rm per,s}^2,$$
	as $n\rightarrow +\infty$. In other words, $\mathcal{R}$ is compact in $H_{\rm per,s}^2$ as requested.
	
	\indent Since the frequency $c$ of the wave given by $(\ref{varphi-stokes3})$ is not constant, we can apply \cite[Theorem 9.1.1]{buffoni-toland} to extend globally the local bifurcation curve given in \eqref{localcurve}. More precisely, there is a continuous mapping
	\begin{equation}\label{globalcurve}
		c \in \left( 2, +\infty \right) \mapsto \varphi \in H_{\rm per,s}^{2},
	\end{equation}
	where $\varphi$ solves equation $(\ref{ode})$ for the case $p=1$.
\end{proof}

			
			\subsection{Existence of periodic minimizers - Case $p=1$}
			
		The existence of periodic waves $\varphi:\mathbb{T}\times\mathbb{T}\rightarrow \mathbb{R}$ that solves $(\ref{ode})$ is established as follows. Let $c>0$ be fixed. First, we prove the existence of a minimizer $\Psi$ of the following minimization problem:
			\begin{equation}
				\label{infB}
				q = \inf_{u\in Y_\lambda} B_c(u), \quad
				B_c(u) = \frac{1}{2}\int_{\mathbb{T}\times\mathbb{T}}|\nabla u|^2+c|u|^2 dxdy,
			\end{equation}
			in the constrained set
			\begin{equation}
				\label{Y-constraint}
				Y_{\lambda} = \left\{u\in H_{\rm per,s}^1;\ \int_{\mathbb{T}\times\mathbb{T}} |u|^3 dxdy = \lambda \right\}.
			\end{equation}
			\indent After that, we use Lagrange multipliers to show that the Euler--Lagrange equation for
			$(\ref{infB})$ and $(\ref{Y-constraint})$ is equivalent to the stationary equation 
			\begin{equation}\label{ode-wave}
			-\Delta \Psi+c\Psi -|\Psi|\Psi=0.	
				\end{equation}
			Finally, we use standard maximum principle to show that the solution $\Psi$ 
			of the minimization problem $(\ref{infB})$ satisfies $\Psi(x,y)=e^{i\theta}\varphi(x,y)$, where $\theta\in\mathbb{R}$ and $\varphi>0$. 
			
\begin{proposition}
\label{minlem}
There exists a ground state
of the constrained minimization problem \eqref{infB}. In other words, there exists $\Psi\in Y_{\lambda}$ satisfying
\begin{equation}\label{minBfunc}
B_c(\Psi)=\inf_{u\in Y_{\lambda}}B_c(u).
\end{equation}
\end{proposition}

\begin{proof} Since $c>0$, we see that $B_c(u)\geq0$ for all $u\in Y_{\lambda}$. This fact enables us to consider a minimizing sequence $(u_n)_{n\in\mathbb{N}}$ such that $B_c(u_n)\rightarrow q$ and $\int_{\mathbb{T}\times\times\mathbb{T}}|u_n|^3dxdy=\lambda$. Standard compactness arguments give us that the infimum is attained at $\Psi\in Y_\lambda$.

\end{proof}		

\begin{proposition}
\label{cor-existence}
There exists a solution to the periodic boundary-value problem
(\ref{ode-wave}) with the minimizer profile $\Psi$ obtained in Proposition $\ref{minlem}$. In addition, $\Psi$ can be consider as $\Psi(x,y)=e^{i\theta}\varphi(x,y)$, where $\theta\in\mathbb{R}$ and $\varphi>0$.
\end{proposition}

\begin{proof}
By Lagrange's multiplier theorem, the constrained minimizer $\Phi \in Y_{\lambda}$ in Proposition \ref{minlem} satisfies
the stationary equation
\begin{equation}
\label{lagrange}
-\Delta \Psi+c\Psi=C_1|\Psi|\Psi
\end{equation}
for some constant $C_1$. A standard scaling argument enables us to consider $C_1=1$.\\
\indent Next, consider $\Psi=\psi_1+i\psi_2$ and define $\Phi=|\psi_1|+i|\psi_2|$. We see that $|\Psi|=|\Phi|$ and $|\nabla\Psi|=|\nabla\Phi|$, and both two facts imply that $\Phi$ is a also a minimizer of the problem $(\ref{infB})$. We see that $(\ref{lagrange})$ is also satisfied for $\Phi$ in place of $\Psi$. Separating equation $(\ref{lagrange})$ into real and imaginary parts, we obtain that $|\psi_1|$ and $|\psi_2|$ satisfies the following equations

\begin{equation}\label{psi11}
-\Delta |\psi_1|+c|\psi_1|=|\Psi||\psi_1|,
\end{equation}
and
\begin{equation}\label{psi22}
-\Delta |\psi_2|+c|\psi_2|=|\Psi||\psi_2|.
\end{equation}
Equalities $(\ref{psi11})$ and $(\ref{psi22})$ give us, since $c>0$, that $|\psi_1|$ and $|\psi_2|$ are elements in $H_{\rm per,s}^2\hookrightarrow C_{\rm per,s}$. Here $C_{\rm per,s}$ indicates the space of periodic continuous function $f$ satisfying $f(-x,y)=f(x,y)$ and $f(x,-y)=f(x,y)$ for all $(x,y)\in\mathbb{T}\times\mathbb{T}$. Using the strong maximum principle, we obtain, since $|\psi_1|,|\psi_2|\geq0$, that $|\psi_1|,|\psi_2|>0$ over $\mathbb{T}\times\mathbb{T}$. This last fact establishes that $\psi_1$ and $\psi_2$ do not change their sign over $\mathbb{T}\times\mathbb{T}$. We can assume, without loss of generality that both $\psi_1$ and $\psi_2$ are positive over $\mathbb{T}\times\mathbb{T}$.\\
\indent We claim that $\psi_1=\alpha\psi_2$ for some $\alpha\in\mathbb{R}$. Suppose that such an $\alpha\in\mathbb{R}$ does not exist. Define the real-valued function $z=\psi_1-\psi_2$. By linearity and the fact that $\psi_1$ and $\psi_2$ also solve equation $(\ref{psi11})$ and $(\ref{psi22})$, respectively, we obtain that function $z$ solves the equation
\begin{equation}\label{z}
-\Delta z+cz=|\Psi|z.
\end{equation}
It is important to mention that we can suppose, without loss of generality, that $\psi_1 \perp \psi_2$, since ${\psi_1, \psi_2}$ is linearly independent and both functions satisfy the linear equations $(\ref{psi11})$ and $(\ref{psi22})$, respectively. Since $(z, \psi_1)_{L_{\rm per}^2} = ||\psi_1||_{L_{\rm per}^2}^2 > 0$ and $(z, \psi_2)_{L_{\rm per}^2} = -||\psi_2||_{L_{\rm per}^2}^2 < 0$, we conclude that $z$ changes its sign over $\mathbb{T} \times \mathbb{T}$. Suppose, without loss of generality, that $z>0$. Since $\psi_1$ and $\psi_2$ are strictly positive functions over $\mathbb{T}\times\mathbb{T}$, we obtain, by continuity, that both inner products $(z,\psi_1)_{L_{\rm per}^2}$ and $(z,\psi_2)_{L_{\rm per}^2}$ are positive. This leads a contradiction, since $(z, \psi_2)_{L_{\rm per}^2} = -||\psi_2||_{L_{\rm per}^2}^2<0$. \\
\indent On the other hand, let us consider the (real)  minimization problem
\begin{equation}
				\label{infB1}
				b = \inf_{u\in W_\tau} D(u), \quad
				D(u) = \frac{1}{2}\int_{\mathbb{T}\times\mathbb{T}}|\nabla u|^2-|\Psi|u^2 dxdy,
			\end{equation}
			where $W_{\tau}=\left\{u\in H_{\rm per,s}^1;\ \int_{\mathbb{T}\times\mathbb{T}} u^2 dxdy = \tau \right\}$. A standard compactness argument, shows that the minimum of the problem $(\ref{infB1})$ is attained in a periodic real function $v$. Moreover, from Lagrange's multiplier theorem, we guarantee the existence of $d\in\mathbb{R}$ such that
\begin{equation}\label{lambda}
-\Delta v+dv=|\Psi|v.
\end{equation}

By defining $w=|v|$ and doing a similar procedure as done for $\Phi$ that $w$ is a also a solution of the minimization problem $(\ref{infB1})$ and it satisfies the following linear equation
\begin{equation}\label{w}
-\Delta w+d w=|\Psi|w.
\end{equation}
In addition, integrating equation $(\ref{w})$ over $\mathbb{T}\times\mathbb{T}$, we obtain automatically that $d>0$. Thus, from the facts that $d>0$ and $w\geq0$, and by applying the strong maximum principle, we obtain that $w>0$, meaning that $v$ does not change its sign. Multiplying $(\ref{lambda})$ by $\psi_1$ and integrating over $\mathbb{T}\times\mathbb{T}$, we obtain after integrating by parts 
\begin{equation}\label{d-c}
		(d-c)\int_{\mathbb{T}\times\mathbb{T}}v\psi_1dxdy=0.
\end{equation}
Since both $\psi_1$ and $v$ are positive periodic functions, we deduce that $d=c>0$. This fact enables us to conclude that the all minimizers of the problem $(\ref{infB1})$ can be considered positive and they solve equation $(\ref{z})$.\\
\indent Next, consider $h=\frac{\tau z}{||z||_{L_{per}^2}}$. Multiplying equation $(\ref{z})$ by $\frac{z}{||z||_{L_{per}^2}^2}$ and integrating the result over $\mathbb{T}\times\mathbb{T}$, we obtain $D(h)=-\frac{c\tau}{2}=b$. This implies that $h$ is also a minimizer of problem $(\ref{infB1})$. Using a similar analysis as above, we can conclude that $|h|$ is also a minimizer, and by the strong maximum principle, we obtain that $|h| > 0$, meaning that $h$ does not change its sign over $\mathbb{T} \times \mathbb{T}$, nor does $z$. This leads to a contradiction because we claimed that $z=\psi_1-\psi_2$ changes its sign. Thus, we have $\psi_1 = \alpha \psi_2$ with $\alpha > 0$. Finally, since $\Psi = \psi_1 + i\psi_2 = (\alpha + i)\psi_2$, we can consider $\theta \in \mathbb{R}$ such that $\Psi = e^{i\theta}\varphi$ with $\varphi > 0$, as desired. This completes the proof of the proposition.
\end{proof}

\begin{remark}\label{remark3} Let $c>2$ be fixed. There exists a non-constant positive and periodic solution $\phi\in H_{\rm per,s}^2$ of the equation $(\ref{ode})$ is a solution of the minimization problem $(\ref{infB})$. In fact, for $\lambda>0$, define $\lambda=\int_{\mathbb{T}\times\mathbb{T}}\phi^3dxdy$. For the fixed values of $c$ and $\lambda$, it follows by Proposition $\ref{cor-existence}$ that there exists a positive and double periodic function $\varphi\in H_{\rm per,s}$ such that 
	\begin{equation}\label{comparesol}
		B_{c}(\varphi)=\inf_{u\in Y_{\lambda}}B_c(u).
		\end{equation}
	On the other hand, multiplying equation $(\ref{ode})$ with $\phi$ in place of $\varphi$ by $\phi$ and integrating once, we obtain
	\begin{equation}\label{comparesol1}
		\frac{1}{2}\int_{\mathbb{T}\times\mathbb{T}}|\nabla \phi|^2+c\phi^2=\frac{1}{2}\int_{\mathbb{T}\times\mathbb{T}}\phi^3dxdy=\frac{\lambda}{2}.
		\end{equation}
		Since the minimizer $\varphi$ is also a solution of the equation $(\ref{ode})$, we deduce $B_{c}(\varphi)=\frac{\lambda}{2}$. Gathering this last information with $(\ref{comparesol})$ and $(\ref{comparesol1})$, we obtain $B_{c}(\varphi)=B_{c}(\phi)$ as desired.
\end{remark}

\begin{proposition}\label{simpleKernel2even}
	Let $c>2$ be fixed. If $\varphi \in H^2_{\rm per,s}$ is the periodic minimizer given by Proposition \ref{theorem1even}, then $n(\mathcal{L}_1)=n(\mathcal{L})=1$, ${\rm n}(\mathcal{L}_2)=0$, and  ${\rm z}(\mathcal{L}_2)=1$.
\end{proposition}

\begin{proof}Thanks to the variational formulation \eqref{infB}, Proposition $\ref{minlem}$, and Lemma $\ref{theorem1even}$,  we obtain $\varphi$ as a positive non-constant minimizer of $E(u)$ in \eqref{Eu} for every $c>2$ subject to one constraint. In the sector $L_{\rm per,s}^2$ of $L_{\rm per}^2$, since $\mathcal{L}$ in $(\ref{matrixop})$ is the Hessian operator for the functional $G(u)=E(u)+cF(u)$, it follows by the min-max principle \cite[Theorem XIII.2]{ReedSimon} that
	$
	{\rm n}(\mathcal{L}) \leq 1.
	$
	Since $\mathcal{L}_1\varphi=-\varphi^2<0$ and $(\mathcal{L}_1\varphi,\varphi)_{L^2_{per}}=-\int_{\mathbb{T}\times\mathbb{T}}\varphi^3dxdy<0$, we have ${\rm n}(\mathcal{L}_1) \geq 1$. The operator $\mathcal{L}$ in $(\ref{matrixop})$ is diagonal and thus ${\rm n}(\mathcal{L}_1)={\rm n}(\mathcal{L})=1$, so that ${\rm n}(\mathcal{L}_2)=0$. Next, we see that $\mathcal{L}_2\varphi=0$ with ${\rm n}(\mathcal{L}_2)=0$. It follows by Krein Ruttman's theorem that, in addition to 
	 $\varphi>0$, the zero eigenvalue for $\mathcal{L}_2$ results to be simple in the sector $L_{\rm per,s}^2$ of $L_{\rm per}^2$. 
\end{proof}

\begin{proposition}\label{theorem1even}
	Let $c>0$ be fixed. Let $\varphi \in H^2_{\rm per,s}$ be the real-valued periodic minimizer  given by $(\ref{infB})$. If $c \in \left(0, 2\right]$ then $\varphi$ is the constant solution and if $c \in \left(2, +\infty \right)$ then $\varphi\in H^2_{\rm per,s}$ is a periodic non-constant minimizer for the equation $(\ref{ode})$. In particular, for $c\rightarrow2^{+}$, the solution in $(\ref{varphi-stokes31})$ solves the minimization problem $(\ref{infB})$.
\end{proposition}
\begin{proof}
	\indent Since the solution can be constant, we need to avoid this case in order to guarantee that the minimizer has a non-constant profile. First, we see that the positive constant solution of the equation \eqref{ode} is $\varphi\equiv c$ and the operator  $\mathcal{L}_1$ in $(\ref{L1L2})$ is then given by
	$
	\mathcal{L}_1= -\Delta-c.
	$
	In addition, it follows that ${\rm n}(\mathcal{L}_1)=1$ if, and only if, $c \in \left(0, 2\right]$. On the other hand, we have to notice that $\varphi = c$ is not a minimizer of \eqref{infB} for $c > 2$ since in this case we have ${\rm n}(\mathcal{L}_1) > 1$ (for $c>2$ we see that $\varphi$ is a periodic minimizer of $B_c(u)$ restricted only to one constraint and it is expected that $n(\mathcal{L}_1)\leq1$ since $n(\mathcal{L})\leq1$). We conclude that the constant solution $\varphi=c$ is a minimizer of \eqref{infB} for $c\in \left(0, 2\right]$ and for $c \in \left(2, +\infty \right)$, solution  $\varphi$ is a nonconstant minimizer.
\end{proof}

	\subsection{Existence of periodic solutions - Case $p\geq2$} The first part of the proof of Proposition $\ref{propstokes1}$, that is, the part where we guarantee the existence of small-amplitude periodic waves for the case $p=1$ can be extended for all integers $p\geq 2$. In fact,  let $\mathsf{F}: H_{\rm per,s}^{2} \times \mathbb{R} \rightarrow L_{\rm per,s}^2$ be the smooth map defined by
\begin{equation}\label{F-lyapunov12}
	\mathsf{F}(g,c) = -\Delta g+cg-g^{p+1}.
\end{equation}
Again, we obtain that $\mathsf{F}(g,c) = 0$ if, and only if, $g \in H_{\rm per,s}^{2}$ satisfies \eqref{ode} with corresponding frequency $c \in \mathbb{R}$. For $c_0>0$, the Fr\'echet derivative of the function $\mathsf{F}$ with respect to the first variable at the equilibrium point $(c_0,c_0^{1/p})$ is then given by
\begin{equation}\label{Dg2}
	D_g \mathsf{F}(c_0, c_0^{1/p}) f = (-\Delta - pc_0) f.
\end{equation}

The nontrivial kernel of $D_g\mathsf{F}(c_0, c_0^{1/p})$ is given by
\begin{equation}
	{\rm Ker}D_g \mathsf{F}(c_0, c_0^{1/p}) = {\rm span}\{\tilde{\varphi}_{k,j}\},
\end{equation}
where $\tilde{\varphi}_{k,j}(x,y) = \cos(kx)\cos(jy)$ and $k,j\in\mathbb{Z}$ satisfy $c_0=\frac{k^2+j^2}{p}$. Since our intention is to obtain periodic functions defined over the set $\mathbb{T} \times \mathbb{T}$, we consider $k = j = 1$. This implies that $c_0 = \frac{2}{p}$, and, using similar ideas as in Proposition~\ref{propstokes1}, we obtain the existence of an open interval $I$ to the right of and sufficiently close to $\frac{2}{p}$, an open ball $B(0, r) \subset H_{{\rm per}, s}^2$ for some $r > 0$, and a smooth mapping
\begin{equation}\label{localcurve1}c \in I \mapsto \varphi= \varphi_c \in B(0,r) \subset H_{\rm per,s}^{2}\end{equation}
such that $\mathsf{F}(\varphi,c) = 0$ for all $c \in I$ and $\varphi\in B(0,r)$. A continuous curve of global solutions existing for all $c>\frac{2}{p}$ can be determined using ideas similar to those in Proposition \ref{propstokes2}. Moreover, for any $p \geq 2$, the results established in Propositions $\ref{minlem}$–$\ref{theorem1even}$ continue to hold under suitable modifications.
\begin{remark}\label{expsol3} It is possible to determine explicit formulas for the small-amplitude periodic waves in the case $p \geq 2$ using the similar recurrence formulas in $(\ref{recurrence})$ adapted to the general case. Indeed, consider the expansions
		\begin{equation}\label{stokes1313}\varphi(x,y) =\left(\frac{2}{p}\right)^{1/p}+\sum_{n=1}^{+\infty}a^n\varphi_n(x,y)\ \ \mbox{and}\ \ c = \frac{2}{p}+ \sum_{n=1}^{+\infty}c_{2n}a^{2n}.\end{equation}
	From $(\ref{ode})$ and $(\ref{stokes1313})$, we obtain that $c_{2n}$ and $\varphi_n$ are uniquely determined since they satisfy the following recurrence relations
	\begin{equation}\label{recurrence11}\left\{\begin{array}{llllll}
			\mathcal{O}(a)\ :  &-\displaystyle\Delta\varphi_1-2\varphi_1=0,\\
			\mathcal{O}(a^2):  &-\displaystyle \Delta \varphi_2-2\varphi_2+\left(\frac{2}{p}\right)^{1/p}c_2-\left(\frac{2}{p}\right)^{(p-1)/p}
			\left(\begin{array}{cccc}p+1 \\ 2\end{array}\right)\varphi_1^2=0,\\
			\mathcal{O}(a^3): &-\displaystyle\Delta\varphi_3-2\varphi_3+c_2\varphi_1-	p(p+1)\left(\frac{2}{p}\right)^{(p-1)/p}\varphi_1\varphi_2\\
			&-\displaystyle\left(\frac{2}{p}\right)^{(p-2)/p}\left(\begin{array}{cccc}p+1 \\ 3\end{array}\right)\varphi_1^3=0,
		\end{array}\right.\end{equation}	
	where, for positive integers $m$ and $q$ satisfying $m\geq q$, we have $\left(\begin{array}{cccc}m \\ q\end{array}\right)=\frac{m!}{q!(m-q)!}$. Again, we deduce that $\varphi_1(x,y)=\cos(x)\cos(y)$ satisfies the equation containing the term $\mathcal{O}(a)$. Next, we can solve the inhomogeneous equation for $\mathcal{O}(a^2)$ to obtain

\begin{equation}\label{stokes212}\varphi_2(x,y)=\left( \frac{\left( \frac{2}{p} \right)^{1/p} c_2 - \frac{B}{4}}{2} \right) + \frac{B}{8} \cos(2x) + \frac{B}{8} \cos(2y) + \frac{B}{24} \cos(2x)\cos(2y),
\end{equation} 
where $B=\left( \frac{2}{p} \right)^{(p-1)/p}\frac{(p+1)!}{2!(p-1)!}$. The value $c_2$ is a constant to be determined. In fact, the inhomogeneous equation at $\mathcal{O}(a^3)$ admits a solution $\varphi_3$ if, and only if, the right-hand side is orthogonal to $\varphi_1$, which selects uniquely the
correction \begin{equation}\label{c21}c_2 = -\frac{1}{p} \left( \frac{p(p+1) }{192} \left( \frac{2}{p} \right)^{2(p-1)/p}\frac{(p+1)!}{(p-1)!}  + \frac{3}{32} \left( \frac{2}{p} \right)^{(p-2)/p}\frac{(p+1)!}{(p-2)!} \right).\end{equation}

Function $\varphi_3(x,y)$ has a complicated expression involving the power $p$, and it will be omitted. It is important to mention that, for our purposes, $\varphi_1$, $\varphi_2$ and the value of $c_2$ are sufficient for the spectral stability analysis.

\end{remark}

\section{Spectral stability}\label{stasec}

Before proving our main spectral stability result, we need to prove additional properties of the operators appearing in \eqref{L1L2}. 

\begin{lemma}\label{bouL2}
	Let $p\geq1$ be a fixed integer. Consider $\varphi$ to be the small-amplitude periodic waves obtained in Proposition \ref{propstokes1} and Remark $\ref{expsol3}$. There exists $\delta>0$ such that
	\begin{equation}\label{a1}
		(\Ldois Q,Q)_{L_{\rm per}^2}\geq\delta\|Q\|_{L_{\rm per}^2}^2,
	\end{equation}
	for all $Q\in H_{\rm per,s}^2$ satisfying $(Q,\varphi)_{L_{\rm per}^2}=0$.
\end{lemma}
\begin{proof}
	Write $L_{\rm per}^2={\rm span}\{\varphi\}\oplus M$ where $\varphi\perp Q$ for all $Q\in M$. Since $\varphi$ belongs to the kernel of
	$\Ldois$ it follows by \cite[Theorem 6.17]{kato} that the spectrum of the
	part $\Ldois{\mid}_M$ coincides with $\sigma(\Ldois)\setminus\{0\}$.
	Proposition \ref{simpleKernel2even} and the arguments in \cite[page 278]{kato} imply
	that there exists a $\delta>0$ such that $\Ldois\geq\delta$ on $M\cap
	H_{\rm per}^2$. Now, if $Q\in H_{\rm per}^2$ satisfies $(Q,\varphi)_{L_{\rm per}^2}=0$ we have $Q\in M$ and
	the conclusion of the lemma then follows.
\end{proof}

\begin{lemma}\label{bouL3}
Let $p\geq1$ be a fixed integer. Consider $\varphi$ to be the small-amplitude periodic waves obtained in Proposition \ref{propstokes1} and Remark $\ref{expsol3}$.\\
i) There exists an $h \in H_{\rm per,s}^2$ such that $\mathcal{L}_1 h = \varphi$.\\
ii) We have that $(\mathcal{L}_1^{-1}\varphi,\varphi)_{L_{\rm per}^2}<0$.
\end{lemma}
\begin{proof}
We observe that for $c$ in a neighborhood to the right of the critical value $c_0 = \frac{2}{p}$, $\varphi$ is a smooth curve depending on  $c>\frac{2}{p}$. Deriving equation $(\ref{ode})$ with respect to $c$, we obtain
\begin{equation}\label{dervarphi}
	-\Delta \frac{d\varphi}{dc}+c\frac{d\varphi}{dc}-(p+1)\varphi^{p}\frac{d\varphi}{dc}=-\varphi.
	\end{equation}
	By taking $h=-\frac{d\varphi}{dc}$, we obtain the desired result in item $i)$.\\
	\indent Next we prove $ii)$. Since $\mathcal{L}_1$ is a self-adjoint, closed, unbounded operator defined on $L_{\rm per,s}^2$ with domain $H_{\rm per,s}^2$, we deduce that $\Ker(\mathcal{L}_1)^{\bot} = {\rm Range}(\mathcal{L}_1)$, with $\mathcal{L}_1: \Ker(\mathcal{L}_1)^{\bot} \rightarrow \Ker(\mathcal{L}_1)^{\bot}$ being an invertible operator. This implies that the inner product involving the linear operator $\mathcal{L}_1^{-1}$ in $ii)$ makes sense. It remains to calculate them for the small-amplitude periodic waves. In fact, we have 
	\begin{equation}\label{derc1}(\mathcal{L}_1^{-1}\varphi,\varphi)_{L_{\rm per}^2}=-\frac{1}{2}\frac{d}{dc}\int_{\mathbb{T}\times\mathbb{T}}\varphi(x,y)^2dxdy.
	\end{equation}
 Substituting $(\ref{stokes1313})$ and $(\ref{stokes212})$ into $(\ref{derc1})$, we deduce that 
	$$(\mathcal{L}_1^{-1}\varphi,\varphi)_{L_{\rm per}^2}=-\left(\frac{2}{p}\right)^{2/p}\left(c+2-\frac{2}{p}\right)\pi^2+\mathcal{O}(a^2)<0.$$
\end{proof}

\textit{Proof of Theorem $\ref{teo-1}$.} We need to calculate the matrix $V$ whose entries are given by \eqref{V} for any $p\geq1$ integer. In fact, since the small-amplitude periodic wave $\varphi$ is positive, we see, by using a similar idea as in Proposition $\ref{simpleKernel2even}$, that $\Ker(\mathcal{L}_2)={\rm span}\{\varphi\}$. For $c \to \frac{2}{p}^{+}$, we obtain from \eqref{varphi-stokes3} that $a \to 0$, and thus, by \eqref{varphi-stokes31}, we deduce that $\varphi_c \to \left(\frac{2}{p}\right)^{1/p} = \phi_0 = \text{equilibrium solution of \eqref{ode}}$ uniformly in $H_{\rm per,s}^2$. We can use a similar idea as in Proposition $\ref{simpleKernel2even}$ and  Lemma $\ref{theorem1even}$, adapted for a general integer $p \geq 1$, together with continuity arguments and the fact that ${\rm n}(\mathcal{L}) = {\rm n}(\mathcal{L}_1) = 1$, to obtain that for $c \sim \left(\frac{2}{p}\right)^{+}$, we have ${\rm z}(\mathcal{L}_1) \leq {\rm z}(D_g \mathsf{F}(c_0, c_0)) = 1$, where $c_0 = \frac{2}{p}$ and $D_g \mathsf{F}(c_0, c_0^{1/p})$ is given by \eqref{Dg}. Therefore, for small-amplitude periodic waves, we conclude that ${\rm z}(\mathcal{L}) \leq 2$. If ${\rm z}(\mathcal{L}) = 2$, we have $V_{12} = V_{21} = 0$, $V_{11} = ( \mathcal{L}_2^{-1}v_1, v_1)_{L^2_{\rm per}}$, and $V_{22} = ( \mathcal{L}_1^{-1}\varphi, \varphi)_{L^2_{\rm per}}$, where $\Ker(\mathcal{L}_1) = {\rm span}\{v_1\}$. By Lemma \ref{bouL3}-i), we obtain that $\varphi \in \text{Range}(\mathcal{L}_1)$, that is, there exists $h \in H_{\rm per,s}^2$ such that $\mathcal{L}_1 h = \varphi$. This implies that $(v_1, \varphi)_{L^2_{\rm per}} = (v_1, \mathcal{L}_1 \chi)_{L^2_{\rm per}} = (\mathcal{L}_1 v_1, \chi)_{L^2_{\rm per}} = 0$. Consequently, $V_{11} > 0$ by using Lemma \ref{bouL2}. Finally, by Lemma \ref{bouL3}-ii), we obtain that ${\rm n}(V)=1$. Since, by using a similar argument as in Proposition \ref{simpleKernel2even}, we have ${\rm n}(\mathcal{L}) = 1$ for any integer $p \geq 1$, it follows that ${\rm n}(\mathcal{L}) - {\rm n}(V) = 0$, and the periodic waves are spectrally stable in $L_{\rm per,s}^2$, in accordance with the main results of \cite{KapitulaKevrekidisSandstedeII} and \cite{KapitulaKevrekidisSandstedeI}.\\
\indent On the other hand, if ${\rm z}(\mathcal{L}) = 1$, we obtain only that $V = (\mathcal{L}_1^{-1}\varphi, \varphi)_{L^2_{\rm per}}$. By Lemma $\ref{bouL3}$-ii), we still find that $(\mathcal{L}_1^{-1}\varphi, \varphi)_{L^2_{\rm per}} < 0$, and the periodic wave $(\varphi, 0)$ is spectrally stable for any integer $p\geq1$.

\begin{flushright}
	$\blacksquare$
\end{flushright}

\begin{remark}\label{remstab} Some important facts deserve to be mentioned concerning the orbital stability of small-amplitude periodic waves for the equation $(\ref{nls})$. Indeed, if ${\rm z}(\mathcal{L})=2$, Theorem $\ref{teo-1}$ can not be used to conclude the stability of $(\varphi,0)$ with respect to the orbit generated by rotations and translations
	\begin{equation}\label{orbit}
		\mathcal{O}_\varphi =  \left\{ \left(
		\begin{array}{cc}
			\cos{\zeta} & \sin{\zeta} \\
			-\sin{\zeta} & \cos{\zeta}
		\end{array}
		\right) \left(
		\begin{array}{c}
			\varphi(\cdot-r) \\
			0
		\end{array}
		\right); \; \zeta,r \in \R \right\}.
	\end{equation}
In fact, we first need to remove the translational symmetry in $(\ref{orbit})$ since we are working in the symmetric space $L_{\rm per,s}^2$, where translation is not invariant. Second, the orbit given by $(\ref{orbit})$ should be considered only in terms of rotations, and the abstract approach for obtaining orbital stability is presented in \cite{GrillakisShatahStraussI}. However, since we have only one symmetry and, according to \cite{GrillakisShatahStraussI}, the number of symmetries contained in the orbit should be the same as the exact dimension of $\Ker(\mathcal{L})$, we cannot conclude that the small-amplitude periodic wave $\varphi$ for the equation $(\ref{ode})$ is orbitally stable. Now, if ${\rm z}(\mathcal{L}) = 1$, the orbital stability in $H_{\rm per,s}^1$ can be determined using \cite{GrillakisShatahStraussI}, with the orbit $\mathcal{O}_{\varphi}$ in $(\ref{orbit})$ considered only under rotations.

\end{remark}

\section{Remarks on the existence and stability of small-amplitude periodic waves}

\indent In this section, we obtain different small-amplitude periodic waves for the equation $(\ref{ode})$ in the space $H_{\rm per,e}^2$, constituted by even periodic functions in $H_{\rm per}^2$. Indeed, let $c>0$ be fixed. For any integer $p\geq1$, we obtain, by Lemma \ref{lema2} and using similar arguments as done in Proposition \ref{propstokes1}, that 
\begin{equation}\label{dker}
	{\rm Ker}D_g \mathsf{G}(c, c^{1/p}) = {\rm span}\{\cos(x+y),\cos(x-y)\},
\end{equation}
where  $\mathsf{G}: H_{\rm per,e}^{2} \times \mathbb{R} \rightarrow L_{\rm per,e}^2$ is the smooth map defined by
\begin{equation}\label{F-lyapunov1}
	\mathsf{G}(g,c) = -\Delta g+cg-g^{p+1}.
\end{equation}
\begin{remark}\label{obs123} The information concerning the double kernel in $(\ref{dker})$ allows us to conclude that two distinct solutions emanate from the equilibrium point $(c, c^{1/p})$ for $c \sim \left(\frac{2}{p}\right)^{+}$. The lack of uniqueness of the generator of $\Ker D_g \mathsf{G}(c, c^{1/p})$ prevents us from applying the same approach used in Proposition $\ref{propstokes2}$ to extend the local solutions to a global one. However, since the generators are explicitly given in $(\ref{dker})$, we can still deduce the existence of two distinct small-amplitude periodic solutions associated with equation $(\ref{ode})$.
\end{remark}
\subsection{Existence of periodic waves - Case $p=1$.}
	\begin{proposition}\label{propstokes12+}
There exists $a_0 > 0$ such that for all $a \in (0,a_0)$, there are two even, local periodic solutions $\varphi_+$ and $\varphi_{-}$ to the problem \eqref{ode} for the case $p=1$. The periodic waves can be expressed by the following expansions:
\begin{equation}\label{expansion1+}
	\varphi_{+}(x,y) = 2 + a\cos(x+y) + \frac{a^2}{6}\left(1 + \frac{1}{2}\cos(2x+2y)\right) + \frac{1}{192}a^3\cos(3x+3y) + \mathcal{O}(a^4),
\end{equation}
and
\begin{equation}\label{expansion3-}
	\varphi_{-}(x,y) = 2 + a\cos(x-y) + \frac{a^2}{6}\left(1 + \frac{1}{2}\cos(2x-2y)\right) + \frac{1}{192}a^3\cos(3x-3y) + \mathcal{O}(a^4).
\end{equation}
The frequency associated with both solutions $\varphi_+$ and $\varphi_{-}$, given respectively by \eqref{expansion1+} and \eqref{expansion3-}, is
\begin{equation}\label{expansion2+}
	c = 2 + \frac{5}{12}a^2 + \mathcal{O}(a^4).
\end{equation}
		
	\end{proposition}
\begin{proof} To prove this proposition, we need to use the expansion in $(\ref{stokes1})$ and the recurrence formulas in $(\ref{recurrence})$, with the generator $\varphi_1(x,y) = \cos(x + y)$. In fact, similarly to what was done in the final part of Proposition $\ref{propstokes1}$, we obtain that the small-amplitude periodic waves $\varphi$ in this case become
\begin{equation}\label{expansion1}
	\varphi_{+}(x,y)=2+a\cos(x+y)+\frac{a^2}{6}\left(1+\frac{1}{2}\cos(2x+2y)\right)+\frac{1}{192}a^3\cos(3x+3y)+\mathcal{O}(a^4).
	\end{equation}
The frequency is then given by
\begin{equation}\label{expansion2}
	c=2+\frac{5}{12}a^2+\mathcal{O}(a^4).
	\end{equation}
	\indent On the other hand, if one considers $\varphi_1(x,y)=\cos(x-y)$, we deduce
\begin{equation}\label{expansion3}
	\varphi_{-}(x,y)=2+a\cos(x-y)+\frac{a^2}{6}\left(1+\frac{1}{2}\cos(2x-2y)\right)+\frac{1}{192}a^3\cos(3x-3y)+\mathcal{O}(a^4).
\end{equation}	
As expected, the frequency of the wave $\varphi$ is the same as in $(\ref{expansion2})$.
\end{proof}
\begin{remark}\label{remsol+-} By considering \(\lambda > 0\) satisfying 
\begin{equation}\label{cond1}
	\int_{\mathbb{T}\times\mathbb{T}}\varphi_{+}(x,y)^3 \, dx \, dy = \int_{\mathbb{T}\times\mathbb{T}}\varphi_{-}(x,y)^3 \, dx \, dy =  \lambda,
\end{equation}
we obtain, by Remark $(\ref{remark3})$, that \(\varphi_{+}\) and \(\varphi_{-}\) solve the minimization problem \((\ref{infB})\). Moreover, by Proposition \(\ref{simpleKernel2even}\) and a similar approach to that in the proof of Theorem \(\ref{teo-1}\)-i), we find that \({\rm n}(\mathcal{L}) = 1\) and \({\rm z}(\mathcal{L}_2) = 1\) for \(c \sim 2^{+}\) for both waves \(\varphi_{+}\) and \(\varphi_{-}\). A similar analysis as in Lemma \(\ref{bouL3}\)-ii) allows us to conclude that 
\begin{equation}\label{derc121}
	(\mathcal{L}_1^{-1}\varphi_{+},\varphi_{+})_{L_{\rm per}^2} = (\mathcal{L}_1^{-1}\varphi_{-},\varphi_{-})_{L_{\rm per}^2} = -16c\pi^2 + \mathcal{O}(a^2) < 0.
\end{equation}
\end{remark}
\indent For simplicity of notation, we denote both solutions $\varphi_+$ and $\varphi_-$ by $\varphi_{\pm}$.
\subsection{Existence of periodic waves - Case $p\geq2$}
	Let $p\geq2$ be an integer. Using the expansion in $(\ref{stokes1313})$ and the recurrence formulas in $(\ref{recurrence11})$, we conclude that 
	\begin{equation}\label{expansion114}\begin{array}{lllll}
		\varphi_{\pm}(x,y)&=&\displaystyle\left(\frac{2}{p}\right)^{1/p}+a\cos(x\pm y)\\\\
		&+&\displaystyle\left(\frac{c_2}{2}\left(\frac{2}{p}\right)^{1/p}-\frac{1}{4}\left(\frac{2}{p}\right)^{(p-1)/p}\left(\begin{array}{cccc}p+1 \\ 2\end{array}\right)\right)\\\\
		&+&\displaystyle\frac{p(p+1)}{24}\left(\frac{2}{p}\right)^{(p-1)/p}a^2\cos(2x\pm 2y)\\\\
		&+&\displaystyle\mathcal{O}(a^3).
\end{array}	\end{equation}
	The frequency is then given by
	\begin{equation}\label{expansion224}
		c=\frac{2}{p}+c_2a^2+\mathcal{O}(a^4),
	\end{equation}
	where $c_2$ is given by $$c_2 = \frac{5}{48} p(p+1)^2 \left( \frac{2}{p} \right)^{\frac{2(p-1)}{p}} + \frac{1}{8} (p-1)(p+1) \left( \frac{2}{p} \right)^{\frac{p-2}{p}}.$$
	\indent As in Remark $\ref{remsol+-}$, we see that $\varphi_{\pm}$ in $(\ref{expansion114})$ satisfies the minimization problem $(\ref{infB2})$. Thus, we obtain $n(\mathcal{L}_1) = n(\mathcal{L}) = 1$, ${\rm n}(\mathcal{L}_2) = 0$, and ${\rm z}(\mathcal{L}_2) = 1$ for these waves. In addition, we have 
	\begin{equation}\label{derc1212}
		(\mathcal{L}_1^{-1}\varphi_{+},\varphi_{+})_{L_{\rm per}^2} = (\mathcal{L}_1^{-1}\varphi_{-},\varphi_{-})_{L_{\rm per}^2} = -\left(\frac{2}{p}\right)^{2/p}(c+2)\pi^2 + \mathcal{O}(a^2) < 0.
	\end{equation}
	\subsection{Spectral stability of periodic waves.}
\indent The following result establishes the spectral stability of the periodic waves $\varphi_{\pm}$ for any integer $p \geq 1$.
\begin{proposition}\label{teo-1+}
	Consider the small-amplitude periodic wave solution $\varphi_{\pm}$ given by the expressions  $(\ref{expansion114})-(\ref{expansion224})$. The wave $\Psi = (\varphi_{\pm}, 0)$ is spectrally stable for all $c\sim\frac{2}{p}^{+}$ in the sense of Definition $\ref{def-spectralstability}$, in the space $L_{\rm per,e}^2$.
	\end{proposition}

\begin{proof} We discuss the spectral stability of the waves  $\varphi_{\pm}$ restricted to the space $L_{\rm per,e}^2$. In fact, from $(\ref{dker})$, we obtain by continuity and a similar analysis as made in the proof of Theorem $\ref{teo-1}$, that ${\rm z}(\mathcal{L}_1)\leq 2$ for $c\sim \frac{2}{p}^{+}$. Since ${\rm z}(\mathcal{L}_2)=1$, we conclude that ${\rm z}(\mathcal{L})\leq 3$, and the matrix $V$ whose entries are given by $(\ref{V})$ has, in the worst case, order $3$. Consider, for $c\sim \frac{2}{p}^{+}$, $\{v_{1,\pm},v_{2,\pm}\}$ as an orthogonal basis for $\Ker(\mathcal{L}_1)$ for both cases $\varphi_{+}$ and $\varphi_{-}$. We have for instance that $V_{12}=(\mathcal{L}^{-1}J(v_{1,\pm},0),J(v_{2,\pm},0))_{L_{\rm per}^2}=(\mathcal{L}_2^{-1}v_{1,\pm},v_{2,\pm})_{L_{\rm per}^2}=(\mathcal{L}_2^{-1}v_{2,\pm},v_{1,\pm})_{L_{\rm per}^2}=V_{21}$, and the matrix $V$ then becomes
\begin{equation}\label{matrixV33}V=\left[\begin{array}{ccccc}(\mathcal{L}_2^{-1}v_{1,\pm},v_{1,\pm})_{L_{\rm per}^2}& & (\mathcal{L}_2^{-1}v_{1,\pm},v_{2,\pm})_{L_{\rm per}^2} & & 0\\\\
(\mathcal{L}_2^{-1}v_{2,\pm},v_{1,\pm})_{L_{\rm per}^2}	& & (\mathcal{L}_2^{-1}v_{2,\pm},v_{2,\pm})_{L_{\rm per}^2} & & 0\\\\
0 & & 0 & & (\mathcal{L}_1^{-1}\varphi_{\pm},\varphi_{\pm})_{L_{\rm per}^2}
\end{array}
\right].\end{equation}
By Lemma $\ref{bouL2}$, we see that $$(\mathcal{L}_2^{-1}v_{i,\pm}, v_{i,\pm})_{L_{\rm per}^2} > 0,\ \ \mbox{for}\ i = 1, 2.$$ This fact occurs since $(v_{i,\pm}, \varphi_{\pm})_{L_{\rm per}^2} = (\mathcal{L}_1 v_{i,\pm}, h)_{L_{\rm per}^2} = 0$, where $h \in H_{\rm per,e}^2$ satisfies $\mathcal{L}_1 h = \varphi_{\pm}$. Thus, we have
\begin{equation}\label{detV}
\det(V)=\det(M)(\mathcal{L}_1^{-1}\varphi_{\pm},\varphi_{\pm})_{L_{\rm per}^2},
\end{equation}
where $M$ is the matrix given by 
\begin{equation}\label{matrixM22}M=\left[\begin{array}{ccccc}(\mathcal{L}_2^{-1}v_{1,\pm},v_{1,\pm})_{L_{\rm per}^2}& & (\mathcal{L}_2^{-1}v_{1,\pm},v_{2,\pm})_{L_{\rm per}^2}\\\\
		(\mathcal{L}_2^{-1}v_{2,\pm},v_{1,\pm})_{L_{\rm per}^2}	& & (\mathcal{L}_2^{-1}v_{2,\pm},v_{2,\pm})_{L_{\rm per}^2}
	\end{array}
	\right].\end{equation}
\indent Assume that $\det(M) \geq 0$. Since $(\mathcal{L}_2^{-1}v_{1,\pm}, v_{1,\pm})_{L_{\rm per}^2}$ and $(\mathcal{L}_2^{-1}v_{2,\pm}, v_{2,\pm})_{L_{\rm per}^2}$ are positive numbers, we obtain that ${\rm n}(V) = 1$, and the wave $\varphi_{\pm}$ is spectrally stable in $L_{\rm per,e}^2$. Now, if $\det(M) < 0$, we see that $M$ has a positive and a negative eigenvalue, and thus, since $(\mathcal{L}_1^{-1}\varphi_{\pm}, \varphi_{\pm})_{L_{\rm per}^2} < 0$, we obtain that ${\rm n}(V) = 2$. However, this scenario cannot occur since ${\rm n}(\mathcal{L}) = 1$, and by $(\ref{krein})$, we have a contradiction. Therefore, $\det(M)\geq0$ and both $\varphi_{+}$ and $\varphi_{-}$ are spectrally stable in $L_{\rm per,e}^2$.\\
\indent The spectral stability in the cases ${\rm z}(\mathcal{L})=1$ and ${\rm z}(\mathcal{L})=2$ is determined similarly to the proof of Theorem $\ref{teo-1}$ and by using $(\ref{derc1212})$. Therefore, further details are omitted. 
\end{proof}

	\begin{remark}
	As in Remark $\ref{remstab}$, we cannot conclude (using a similar analysis) the orbital stability of the small-amplitude periodic waves $\varphi_{\pm}$ given by $(\ref{expansion1})-(\ref{expansion3})$ and $(\ref{expansion114})$ in the cases ${\rm z}(\mathcal{L})=2$ and ${\rm z}(\mathcal{L})=3$. The orbital stability can be concluded, using \cite{GrillakisShatahStraussI}, only in the case ${\rm z}(\mathcal{L})=1$.
	\end{remark}

\subsection*{Funding}
F. Natali is partially supported by CNPq/Brazil (grant 303907/2021-5). \\

\subsection*{Data availability} Data sharing not applicable to this article as no datasets were generated or analysed during the current study.\\

\begin{flushleft}
	\textbf{Declarations.}
\end{flushleft}
\subsection*{Conflict of interest} The author declares that he has no conflict
of interest.
\subsection*{Ethics declaration} Not applicable.

\end{document}